\documentclass[a4paper,11pt]{amsart}

\usepackage[utf8]{inputenc}
\usepackage[T1]{fontenc}
\usepackage{soul}
\usepackage{amssymb}
\usepackage{caption}
\usepackage[dvipsnames]{xcolor}
\usepackage{enumitem}
\usepackage{graphicx}
\usepackage{ifthen}
\usepackage{lmodern}
\usepackage{mathtools}
\usepackage[]{microtype}
\UseMicrotypeSet[protrusion]{basicmath}
\usepackage{nicefrac}
\usepackage{stmaryrd}
\SetSymbolFont{stmry}{bold}{U}{stmry}{m}{n}
\usepackage{tabularx}
\usepackage{tikz}
\usepackage{verbatim}
\usepackage{xspace}
\usepackage[all]{xy}

\usepackage[initials, alphabetic]{amsrefs}

\usepackage[hidelinks]{hyperref}

\numberwithin{equation}{section}
\newtheorem{theorem}{Theorem}[section]

\newtheorem{lemma}[theorem]{Lemma}
\newtheorem{proposition}[theorem]{Proposition}

\newtheorem*{theorem*}{Theorem}

\theoremstyle{definition}
\newtheorem{definition}[theorem]{Definition}

\newtheorem*{TP0}{The Trace Problem}

\theoremstyle{remark}
\newtheorem{claim}[theorem]{Claim}

\DeclareMathOperator{\LC}{LC}

\newcommand{\bbC}{\mathbb{C}}

\newcommand{\bbF}{\mathbb{F}}

\newcommand{\cK}{\mathcal{K}}

\newcommand{\cM}{\mathcal{M}}

\newcommand{\bbN}{\mathbb{N}}

\newcommand{\cN}{\mathcal{N}}

\newcommand{\cP}{\mathcal{P}}

\newcommand{\cR}{\mathcal{R}}

\newcommand{\cS}{\mathcal{S}}

\newcommand{\bbT}{\mathbb{T}}

\newcommand{\cU}{\mathcal{U}}

\newcommand{\cZ}{\mathcal{Z}}

\newcommand{\wstar}{$W^\ast$}

\newcommand{\e}{\varepsilon}

\renewcommand{\phi}{\varphi}

\title{Continuous Selection of Unitaries in II$_1$ Factors}
\date{}

\thanks{I. F.  was partially supported by NSERC. A. V.  was supported by the Deutsche Forschungsgemeinschaft (DFG, German Research Foundation) under Germany’s Excellence Strategy EXC 2044--390685587, Mathematics Münster: Dynamics--Geometry--Structure, through SFB 1442 and ERC Advanced Grant 834267--AMAREC}

\author{Ilijas Farah}
\address{Ilijas Farah, Department of Mathematics and Statistics, York University, Toronto, Ontario M3J 1P3, Canada and Matemati\v{c}ki Institut SANU, 
	Kneza Mihaila 36,
	11000 Beograd, P. P. 367, Serbia}
\email{ifarah@yorku.ca}
\urladdr{https://ifarah.mathstats.yorku.ca}
\author{Andrea Vaccaro}
\address{Andrea Vaccaro, Mathematisches Institut, Fachbereich Mathematik und Informatik der
Universit\"at M\"unster, Einsteinstrasse 62, 48149 M\"unster, Germany.}
\email{avaccaro@uni-muenster.de}
\urladdr{https://sites.google.com/view/avaccaro}

\begin{document}
\maketitle

\begin{abstract}
We prove continuous-valued analogues of the basic fact that Murray--von Neumann subequivalence of projections in II$_1$ factors is completely determined by tracial evaluations. We moreover use this result to solve the so-called \emph{trace problem} in the case of factorial trivial $W^\ast$-bundles whose base space has covering dimension at most 1. Our arguments are based on applications to von Neumann algebras of a continuous selection theorem due to Michael.
\end{abstract}

\section{Introduction}

It is a basic fact that if $\cM$ is a II$_1$ factor with trace $\tau$ and $p, q \in \cM$ are projections, then $p$ is \emph{subequivalent} to $q$---i.e. there exists $v \in \cM$ such that $v^*v = p$ and $vv^* \le q$---if and only if $\tau(p) \le \tau(q)$.  In this note we consider a continuous-valued analogue of this property: take the 2-norm
\[
\| a \|_{2, \tau} \coloneqq \tau(a^*a)^{1/2}, \quad a \in \cM,
\]
and suppose that $(p(t))_{t \in [0,1]}$ and $(q(t))_{t \in [0,1]}$ are $\|\cdot \|_{2,\tau}$-continuous paths of projections in $\cM$, such that $\tau(p(t)) \le \tau(q(t))$ for all $t \in [0,1]$. Is there a $\|\cdot \|_{2,\tau}$-continuous function $v \colon [0,1] \to \cM$ such that $v(t)^*v(t) = p(t)$ and $v(t)v(t)^*\le q(t)$ for all $t \in [0,1]$?

We employ a continuous selection theorem due to Michael from \cite{Michael:selection2} to give a positive answer, covering in fact a larger class of cases. The maximum generality is obtained for II$_1$ factors of the form $\cN \bar \otimes L(\bbF_\infty)$ where $\cN$ is a finite factor and $L(\bbF_\infty)$ is the factor generated by the free group with infinitely many generators.

\begin{theorem}\label{thm:continuous_comparison}
Let $X$ be a compact Hausdorff space, let $(\cM, \tau)$ be a II$_1$ factor. Suppose that $p, q \colon X \to \cM$ are projection-valued $\|\cdot \|_{2,\tau}$-continuous functions and that one of the following conditions holds:
\begin{enumerate}[label=(\alph*)]
\item \label{item:cov1} $X$ has covering dimension at most 1.
\item \label{item:Finfty} $X$ has finite covering dimension, and $\cM \cong \cN \bar{\otimes} L(\bbF_\infty)$ for a finite factor $\cN$.
\end{enumerate}
Then the following two statements hold.
	\begin{enumerate}
\item \label{item:subequivalence} There exists a $\|\cdot \|_{2,\tau}$-continuous function $v \colon X \to \cM$ such that $v(x)^*v(x) = p(x)$ and $v(x)v(x)^* \le q(x)$ for all $x \in X$ if and only if $\tau(p(x)) \le \tau(q(x))$ for all $x \in X$.
\item \label{item:equivalence} There exists a unitary-valued $\|\cdot \|_{2,\tau}$-continuous function $v \colon X \to \cM$ such that $v(x) p(x) v(x)^* = q(x)$ for all $x \in X$ if and only if $\tau(p(x)) = \tau(q(x))$ for all $x \in X$.
\end{enumerate}
\end{theorem}

Our interest in Theorem \ref{thm:continuous_comparison} was motivated by the so-called \emph{trace problem} for factorial tracially complete $C^\ast$-algebras, which we briefly report here.

Let $(\cM, X)$ be a pair where $\cM$ is a unital $C^\ast$-algebra and $X$ is a compact convex subset of the set $T(\cM)$ of all tracial states of $\cM$ (simply called \emph{traces} henceforth). The pair $(\cM, X)$ is a \emph{tracially complete $C^\ast$-algebra} if the 2-seminorm
\[
\| a \|_{2, X} \coloneqq \sup_{\tau \in X} \tau(a^*a)^{1/2}, \quad a \in \cM,
\]
is a norm on $\cM$, and if the unit ball $\cM_1$ (in the operator norm) is complete with respect to $\| \cdot \|_{2, X}$.\footnote{If $(\cM, X)$ is tracially complete, the requirement that $\cM$ is unital becomes redundant if $T(\cM)$ is assumed to be compact, as shown in \cite[Proposition 3.9]{CCEGSTW}.} A tracially complete $C^\ast$-algebra is \emph{factorial} if moreover $X$ is a face of $T(\cM)$. We refer to \cite{CCEGSTW} for a detailed introduction and study of these objects. The trace problem, in its most general form, asks the following.
\begin{TP0}[{\cite[Question 1.1]{CCEGSTW}}]
Let $(\cM, X)$ be a factorial tracially complete $C^\ast$-algebras. Is $X = T(\cM)$?
\end{TP0}
Note that every tracial von Neumann algebra $(\cM, \tau)$ is a tracially complete $C^\ast$-algebra (with $X = \{ \tau \}$), and it is factorial if and only if $\cM$ is a factor, in which case $T(\cM) = \{ \tau \}$.

In order to contextualize and motivate the trace problem, we briefly recall the notion of (uniform) tracial completion. The tracial completion of a $C^\ast$-algebra $A$ with compact non-empty trace space $T(A)$ is generated, as defined by Ozawa in \cite{Ozawa:dixmier}, by the completion of the closed unit ball $A_1$ with respect to the 2-seminorm
\[
\| a \|_{2, T(A)} \coloneqq \sup_{\tau \in T(A)} \tau(a^*a)^{1/2}, \quad a \in A.
\]

The $C^\ast$-algebra obtained from this process is denoted $\overline{A}^{T(A)}$. All traces in~$A$ have a unique $\| \cdot \|_{2, T(A)}$-continuous extension to $\overline{A}^{T(A)}$, which allows identification of $T(A)$ with a subset of $T(\overline{A}^{T(A)})$. Under this identification, the pair $(\overline{A}^{T(A)}, T(A))$ is a factorial tracially complete $C^\ast$-algebra (see \cite[\S 3.3]{CCEGSTW} for details). The trace problem is therefore asking whether $T(A) = T(\overline{A}^{T(A)})$ or, in other words, if $\overline{A}^{T(A)}$ is tracially complete with respect of its \emph{whole} trace space and if the tracial completion is an idempotent operation.

Going back to our Theorem \ref{thm:continuous_comparison}, a precedent is found in \cite[Theorem E]{CCEGSTW}, where the statement of Theorem \ref{thm:continuous_comparison} is proved for all compact Hausdorff spaces, with no assumption on the covering dimension, in case the II$_1$ factor $\cM$ has property $\Gamma$. This result was recently used by Evington in \cite{Evington:trace} to solve the so-called \emph{trace problem} in the case of tracial completions of $\cZ$-stable $C^\ast$-algebras and more generally for all factorial tracially complete $C^\ast$-algebras that satisfy \emph{uniform property $\Gamma$} (\cite[Definition 5.19]{CCEGSTW}). This represents the most general result so far, and indeed little is known beyond the class of tracially complete $C^\ast$-algebras with property $\Gamma$ (see \cite{Vaccaro:Wbundle} for some partial results on ultraproducts of $W^\ast$-bundles). 

In this paper, we use Theorem \ref{thm:continuous_comparison} to solve affirmatively the trace problem for a specific class of \emph{trivial $W^\ast$-bundles}, a subclass of the richer family of $W^\ast$-bundles introduced in \cite{Ozawa:dixmier}. Given a compact Hausdorff space $X$ and a tracial von Neumann algebra $(\cM, \tau)$, the \emph{trivial $W^\ast$-bundle} with base space $X$ and fiber $\cM$ is
\[
C_\sigma(X, \cM) \coloneqq \{ a \colon X \to \cM : a \text{ is bounded and $\|\cdot \|_{2,\tau}$-continuous} \}.
\]
Pointwise operations and the supremum norm endow this set with a $C^\ast$-algebra structure. Moreover, every Radon probability measure $\mu \in \textrm{Prob}(X)$ induces a tracial state $\rho_\mu$ on $C_\sigma(X, \cM)$ defined as
\begin{equation} \label{eq:measure}
\rho_\mu(a) = \int_X \tau(a(x)) \, d \mu(x), \quad a \in C_\sigma(X, \cM).
\end{equation}
These tracial states induce a 2-norm on $C_\sigma(X, \cM)$, defined as
\[
\| a \|_{2, X} \coloneqq \sup_{\mu \in \text{Prob}(X)}\rho_\mu(a^*a)^{1/2}, \quad a  \in C_\sigma(X, \cM).
\]
The algebra $C_\sigma(X, \cM)$ is tracially complete with respect this 2-norm, and it is factorial if and only if $\cM$, referred to as the \emph{fiber} of the bundle, is a factor (\cite[\S 3.6]{CCEGSTW}). In this case the trace problem translates into asking whether every trace on $C_\sigma(X, \cM)$ is equal to $\rho_\mu$ for some $\mu \in \textrm{Prob}(X)$.

After Evington's results \cite{Evington:trace}, perhaps the most elementary examples for which the trace problem remained unsolved were $C_\sigma([0,1], \cM)$ and $C_\sigma(\bbT, \cM)$, where $\cM$ is a II$_1$ factor that fails property $\Gamma$, such as the free group factors. Our theorem solves positively the trace problem for these and other cases.
\begin{theorem} \label{thm:trace_problem}
Let $(\cM, \tau)$ be a II$_1$ factor and let $X$ be a compact Hausdorff space with covering dimension at most 1.
Then the trace problem has positive solution for $C_\sigma(X, \cM)$, more precisely every tracial state $\rho \in T(C_\sigma(X, \cM))$ has the form
\[
\rho (a) = \int_X \tau(a(x)) \, d \mu(x), \quad a \in C_\sigma(X, \cM),
\]
for some Radon probability measure $\mu \in \textrm{Prob}(X)$.
\end{theorem}

The proof of Theorem \ref{thm:trace_problem} crucially relies on Theorem \ref{thm:continuous_comparison}. More precisely, from the perspective of $W^\ast$-bundles, Theorem \ref{thm:continuous_comparison} shows that $C_\sigma(X, \cM)$---for $X$ with covering dimension at most 1---has \emph{comparison of projections relative to $X$} in the following sense: in order to determine whether two projections $C_\sigma(X, \cM)$ are equivalent (or whether one is subequivalent to the other), it is sufficient to compare their tracial evaluations for traces in $\{ \rho_{\delta_x} \}_{x \in X}$, where $\delta_x$ is the Dirac measure corresponding to $x \in X$.

If $C_\sigma(X, \cM)$ had real rank zero this would be sufficient to deduce that $T(C_\sigma(X, \cM))$ is the closed convex hull of $\{ \rho_{\delta_x} \}_{x \in X}$, namely $\{ \rho_\mu \}_{\mu \in \text{Prob}(X)}$.  The question whether $C_\sigma(X,\cM)$ always has real rank zero is, to the best of our knowledge, currently open.   Instead,  we argue as in \cite{Evington:trace} and prove that hereditary subalgebras of $C_\sigma(X, \cM)$ contain sufficiently many projections to obtain \emph{comparison of positive contractions relative to $X$} (see Definition \ref{def:comparison}), which suffices to settle the trace problem.

\subsection*{Summary of the paper}
The paper is organized as follows. Section \ref{S.2} is devoted to preliminaries. In Section \ref{S.3} we prove Theorem \ref{thm:continuous_comparison}, while in Section~\ref{S.4} we prove Theorem \ref{thm:trace_problem}. Section~\ref{S.5} is reserved for concluding remarks.


\section{Preliminaries} \label{S.2}
\subsection{Tracial von Neumann algebras}
Given a $C^\ast$-algebra $A$, we denote its unitary group by $\cU(A)$,  the set of its positive elements by $A_+$ and the set of its projections by $\text{Proj}(A)$. Given a tracial von Neumann algebra $(\cM, \tau)$, we always interpret $\cU(\cM)$ as a topological group and $\text{Proj}(\cM)$ as a space with the the topology induced by the 2-norm
\[
\| a \|_{2, \tau} \coloneqq \tau(a^*a)^{1/2}, \, a \in \cM.
\]
On bounded sets, this coincides  with the strong topology  when $\cM$ is seen as an algebra of operators on $L^2(\cM, \tau)$. Given $a \in \cM$ and $\e > 0$, we use the notation $B_\e(a)$ to denote the $\| \cdot \|_{2, \tau}$-open ball centered in $a$ of radius $\e$.

We isolate the following elementary lemma for later use.
\begin{lemma} \label{lemma:ulc}
Let $(\cM, \tau)$ be a tracial von Neumann algebra, $\e > 0$ and $u,v \in \cU(\cM)$ such that $\| u - v \|_{2, \tau} < \e$. Then there exists a $\| \cdot \|$-continuous path $(w_t)_{t \in [0,1]}$ in $\cU(\cM)$ such that $w_0 = u$, $w_1 = v$ and $\| w_t - w_s \|_{2, \tau} < \e$ for all $s,t \in [0,1]$. 
\end{lemma}
\begin{proof}
Fix $\e >0$ and let $u,v \in \cU(\cM)$ be such that $\| u - v \|_{2,\tau} < \e$. By replacing $u$ with $uv^*$, we may assume that $v =1$.
By Borel functional calculus there is a self-adjoint $a \in \cM$ such that $\| a \| \le \pi$ and $u = \exp(i a)$. The path $(w_t)_{t \in [0,1]}$, where $w_t \coloneqq \exp(ita)$,  is $\|\cdot\|$-continuous and as required. Indeed,
\[
| \exp(itx) - 1 |^2 \le | \exp(ix) - 1 |^2 \text{ for all } x \in [-\pi ,\pi],\, t \in [0,1],
\]
which implies $| w_t - w_s |^2 = | w_{t- s} - 1 |^2 \le | u -1 |^2$, and thus $\| w_t - w_s \|_{2, \tau} < \e$ for all $t,s \in [0,1]$.
\end{proof}

\subsection{Michael's continuous selection theorem}
In this subsection we briefly report all the necessary definitions and the statement of Michael's selection theorem from \cite{Michael:selection2} that will be needed to prove Theorems \ref{thm:continuous_comparison} and \ref{thm:trace_problem}.

Fix two topological spaces $X, Y$ and let $\cS$ be a subfamily of nonempty subsets of $Y$. A function $\Phi \colon X \to \cS$ is \emph{lower semicontinous} if for every open set $U \subseteq Y$ the set
\[
\{ x \in X : \Phi(x) \cap U \not = \emptyset \}
\]
is open in $X$.


In what follows let $S^n$ denote the $n$-dimensional sphere. A topological space $Y$ is \emph{$n$-connected} if for all $m \le n$ and every continuous map $f\colon S^m \to Y$ there is a continuous function $F \colon S^m \times [0,1] \to Y$ such that $F(z,0) = f(z)$ and $F(z,1) = F(w,1)$ for all $z,w \in S^m$.
Thus a space is $0$-connected if and only if it is path-connected and it is $1$-connected if and only if every loop is homotopic to a point. 

A family of subsets $\cS$ of $Y$ is \emph{equi-$\LC^n$} (\emph{equi-locally $n$-connected}) if for every $S_0 \in \cS$, every $u\in S_0$, and every open neighborhood $U$ of $u$ there is an open neighborhood $V$ of $u$ such that for all $S \in \cS$ and $m \le n$, every continuous map $f\colon S^m \to V \cap S$ extends to a continuous function $F \colon S^m \times [0,1] \to S \cap U$ such that $F(z,0) = f(z)$ and $F(z,1) = F(w,1)$ for all $z,w \in S^m$.

The main results of this paper follow from applications of a continuous selection principle due to Michael, of which we report a weaker version, sufficient for our needs.
\begin{theorem}[{\cite[Theorem 1.2]{Michael:selection2}}] \label{thm:michael_selection}
Let $X$ be a compact Hausdorff space such that $\dim(X) \le n+1$, and let $Y$ be a complete metric space. Suppose that $\cS$ is an equi-$\LC^n$ family of non-empty, $n$-connected, closed subsets of $Y$, and that $\Phi \colon X \to \cS$ is a lower semicontinous function. Then there exists a continuous function $F \colon X \to Y$ such that $F(x) \in \Phi(x)$ for all $x \in X$.
\end{theorem}

\subsection{Contractibility of unitary groups of II$_1$ factors}
We recall a result due to Popa and Takesaki from \cite{PopaTakesaki:contractible}, which will be essential to verify the hypotheses needed to invoke Theorem \ref{thm:michael_selection} in our arguments (see also the more recent \cite{Ozawa:contract}).

\begin{theorem}[{\cite[Corollary~2]{PopaTakesaki:contractible}}]
\label{thm:popatakesaki}
	For every II$_1$ factor $(\cM,\tau)$ such that either $\cM \cong \cM \bar\otimes \cR$ where $\cR$ is the hyperfinite II$_1$ factor, or  $\cM \cong \cN\bar\otimes L(\bbF_\infty)$ for some finite factor $\cN$, there exists a continuous map $\alpha\colon  [0,\infty) \times \cU(\cM)\to \cU(\cM)$ such that 
	\begin{enumerate}
		\item $\alpha_0(u)=u$ and $\lim_{s\to \infty} \alpha_s(u)=1$ for all $u \in \cU(\cM)$,
		\item $\alpha_s$ is an injective endomorphism for all $s \ge 0$,
		\item $\alpha_s\circ \alpha_t=\alpha_{s+t}$ for all $s,t\geq 0$, 
		\item \label{item4:et} $\|\alpha_s(u)-\alpha_s(v)\|_{2,\tau}<e^{-s} \|u-v\|_{2,\tau}$ for all $u,v \in \cU(\cM)$.
	\end{enumerate}
\end{theorem}

\subsection{Strict comparison}
We record some definitions on strict comparison that will be needed in Section \ref{S.4}. We refer to \cite{ERS:cone, NgRobert:commutators} for a more detailed background of this material.

Let $A$ be a $C^\ast$-algebra and let $\cK$ be the $C^\ast$-algebra of compact operators on $\ell^2(\bbN)$. For $a, b \in (A\otimes \cK)_+$, we write  $a \precsim b$ when $a$ is \emph{Cuntz subequivalent} to $b$, i.e. if there is a sequence $(r_n)_{n=1}^\infty$ in $A \otimes \cK$ such that $\lim_{n \to \infty} r_n b r_n^\ast = a$. Recall that for projections $p,q$,
the relation $p \precsim q$ is equivalent to the existence of $v$ such that $vv^* = p$ and $v^* v \le q$.

We let $QT(A)$ denote the compact set of all lower-semicontinuous \emph{2-quasitraces} $\tau \colon A \to [0, \infty]$ (simply called \emph{quasitraces} from here on; see \cite[\S 4]{ERS:cone} for a precise definition and the topology on $QT(A)$) and recall that quasitraces on $A$ naturally extend to $A \otimes\cK$. We let $T(A)$ denote the set of all \emph{tracial states}, which we abbreviate as \emph{traces}, on $A$.

Given $a \in (A \otimes \cK)_+$ and $\tau \in QT(A)$, define the \emph{dimension function}
\[
d_\tau(a) \coloneqq \lim_{n \to \infty} \tau(a^{1/n}).
\]
It is a standard fact that, for $a,b \in (A \otimes \cK)_+$, $a \precsim b$ implies $d_\tau(a) \le d_\tau(b)$ for all $\tau \in QT(A)$. We say that $A$ has \emph{strict comparison} if $d_\tau(a) < d_\tau(b)$ for all $\tau \in QT(A)$ implies $a \precsim b$.

\begin{definition}[{\cite[Definition 3.1]{NgRobert:commutators}}] \label{def:comparison}
Let $A$ be a $C^\ast$-algebra and $X \subseteq \text{QT}(A)$ be a compact subset. Then $A$ has \emph{strict comparison relative to $X$} if for any $a, b \in (A \otimes \cK)_+$ we have that $a \precsim b$ whenever there is $\eta >0$ such that $d_\tau(a) \le (1 -\eta) d_\tau(b)$ for all $\tau \in X$.
\end{definition}

As noted by Evington in \cite{Evington:trace}, relative strict comparison is a useful notion to verify that certain sets of traces are sufficiently large. Below we state Evington's result restricted to the case of trivial $W^\ast$-bundles, which is the main focus of this note.

We say that a trivial $W^\ast$-bundle \emph{$C_\sigma(X, \cM)$ has comparison relatively to $X$} if it has comparison relatively to $\{ \rho_{\delta_x} \}_{x \in X}$, where $\rho_{\delta_x}$ is the trace on $C_\sigma(X, \cM)$ defined as in \eqref{eq:measure}, obtained from the Dirac measure corresponding to $x \in X$.

\begin{proposition}[{\cite[Proposition 4.2]{Evington:trace}}] \label{prop:comparison}
Let $(\cM, \tau)$ be a II$_1$ factor and let $X$ be a compact Hausdorff space. Suppose that the trivial $W^\ast$-bundle $C_\sigma(X, \cM)$ has strict comparison relative to $X$. Then every tracial state $\rho \in T(C_\sigma(X, \cM))$ has the form
\[
\rho (a) = \int_X \tau(a(x)) \, d \mu(x), \quad a \in C_\sigma(X, \cM),
\]
for some Radon probability measure $\mu \in \textrm{Prob}(X)$.
\end{proposition}
\begin{proof}
By \cite[Proposition 3.6]{CCEGSTW} the pair $(C_\sigma(X, \cM), \{ \rho_\mu \}_{\mu \in \text{Prob(X)}})$, where $\rho_\mu$ is defined as in \eqref{eq:measure}, is a factorial tracially complete $C^\ast$-algebra. If $C_\sigma(X, \cM)$ has strict comparison relative to $X$, it has strict comparison relative to $\{ \rho_\mu \}_{\mu \in \text{Prob(X)}}$, hence the conclusion follows by \cite[Proposition 4.2]{Evington:trace}.
\end{proof}

\section{Comparison of Projections} \label{S.3}
In this section we prove Theorem \ref{thm:continuous_comparison}.

\begin{proof}[Proof of Theorem \ref{thm:continuous_comparison}]
 We first prove \eqref{item:equivalence}. Suppose that $p, q \colon X \to \cM$ are projection-valued continuous functions such that $\tau(p(x)) = \tau(q(x))$ for all $x \in X$. Consider the map $\Phi \colon X \to \cP(\cU(\cM))$, where $\cP(\cU(\cM))$ denotes the power set of the unitary group $\cU(\cM)$, defined as
\[
\Phi(x) \coloneqq \{ u \in \cU(\cM) : u p(x) u^* = q(x) \}, \quad x \in X.
\]

Each $\Phi(x)$ is closed in $\cU(\cM)$, seen as complete metric space with the distance induced by $\| \cdot \|_{2,\tau}$, and non-empty since $\tau(p(x)) = \tau(q(x))$ and each fiber is a II$_1$-factor. For every $x \in X$, fix once and for all some $u_x \in \Phi(x)$. The map
\begin{align*}
\Phi(x) & \to  \cU(\cM ) \cap \{p(x) \}' \\
u &\mapsto u_x^\ast u
\end{align*}
is an isometry. After identifying the group $\cU(\cM ) \cap \{p(x) \}' $ of all unitaries that commute with $p$ with $\cU(p(x) \cM p(x) \oplus p(x)^\perp \cM p(x)^\perp)$, we conclude that
$\Phi(x)$ is isometrically homeomorphic to $\cU(p(x) \cM p(x) \oplus p(x)^\perp \cM p(x)^\perp)$,  for every $x \in X$.

\begin{claim} \label{claim:lsc}
The function $\Phi$ is lower semicontinuous.
\end{claim}
\begin{proof}
Let $U \subseteq \cU(\cM)$ be an open set, suppose that $x \in X$ is such that $\Phi(x) \cap U \not = \emptyset$. Let $u \in \Phi(x) \cap U$ and $\e > 0$ be such that $B_\e(u) \subseteq U$. Fix $0 < \delta < \frac{\e}{2 \sqrt{2}}$. By continuity of $p$ and $q$ there exists an open neighborhood $Z$ of $x$ in $X$ such that
\[
\| up(y)u^* - q(y) \|_{2, \tau} < \delta, \quad y \in Z.
\]

It is a well-known fact that projections in a tracial von Neumann algebra that are equivalent and $\| \cdot \|_{2,\tau}$-close can be conjugated by a unitary that is $\| \cdot \|_{2,\tau}$-close to the unit. More precisely, given $y \in Z$, by \cite[Lemma XIV.2.1]{Takesaki:III} there is $w_y \in \cU(\cM)$ such that $w_y up(y)u^* w_y^* = q(y)$ and
\[
\| 1 - w_y \|_{2, \tau} \le 2\sqrt{2} \| up(y)u^* - q(y) \|_{2, \tau} < \e.
\]
We conclude that $w_yu \in B_\e(u)$ and thus that $w_yu \in \Phi(y) \cap U$, for every $y \in Z$.
\end{proof}

The next claim is needed for the proof under assumption \ref{item:cov1}, that $X$ has covering dimension at most 1.

\begin{claim} \label{claim:equiLC0}
The space $\Phi(x)$ is 0-connected for every $x \in X$ and $\{ \Phi(x) \}_{x \in X}$ is equi-$\LC^0$.
\end{claim}
\begin{proof}
Being 0-connected means being path-connected. Thus every $\Phi(x)$ is 0-connected since it is isometrically homeomorphic to the unitary group of a von Neumann algebra.

To prove that $\{ \Phi(x) \}_{x \in X}$ is equi-$\LC^0$, it is sufficient to show that for every $\e > 0$, $x \in X$ and $u_0, u_1 \in \Phi(x)$ with $\| u_0 - u_1 \|_{2, \tau} < \e$, there is a continuous path $(u_t)_{t \in [0,1]}$ in $\Phi(x)$, of $\| \cdot \|_{2,\tau}$-diameter less than $\e$, from $u_0$ to $u_1$. Indeed, if this holds, given $x \in X$, $\e > 0$, $u \in \Phi(x)$ and $f \colon \{0,1\} \to \Phi(x) \cap B_{\e/4}(u)$ (note that $S^0$ is just a set with two isolated points) with image $u_0$ and $u_1$, then one can define $F \colon \{0,1 \} \times [0,1] \to \Phi(x) \cap B_\e(u)$ extending $f$ by setting $F(1, t) = u_1$ and $F(0,t) = u_t$ for every $t \in [0,1]$.

Let thus $w_0, w_1 \in \cU(\cM) \cap \{p(x)\}'$ be such that $u_i = u_x w_i$ for $i = 0,1$, so in particular $\| w_0 - w_1 \|_{2, \tau} < \e$. Since $ \cU(\cM) \cap \{p(x)\}'$ is isomorphic to $\cU(p(x) \cM p(x) \oplus p(x)^\perp \cM p(x)^\perp)$, by applying  Lemma \ref{lemma:ulc} to the latter,  there is a $\| \cdot \|$-continuous path $(w_t)_{t \in [0,1]}$ in $ \cU(\cM) \cap \{p(x)\}'$  from $w_0$ to $w_1$ such that $\| w_t - w_s \|_{2, \tau} < \e$ for all $s,t \in [0,1]$. The path $(u_t)_{t \in [0,1]}$, where $u_t \coloneqq u_x w_t$ for $t \in [0,1]$, is as desired.
\end{proof}

The following claim is used to prove Theorem~\ref{thm:continuous_comparison} under assumption \ref{item:Finfty}, that $X$ has finite covering dimension, and  $\cM \cong \cN \bar{\otimes} L(\bbF_\infty)$ for a finite factor~$\cN$.
\begin{claim} \label{claim:equiLCn}
Suppose that either {$\cM \cong L(\bbF_\infty)$} or $\cM \cong \cN \bar \otimes L(\bbF_\infty)$ for some II$_1$ factor $\cN$.
Then $\Phi(x)$ is $n$-connected for every $x \in X$ and $\{ \Phi(x) \}_{x \in X}$ is equi-$\LC^n$, for all $n \in \bbN$.
\end{claim}
\begin{proof}
By \cite[p. 519]{Radulescu:Finfty} the fundamental group of $\cM$ is full, thus the corners $p(x) \cM p(x)$ and $p(x)^\perp \cM p(x)^\perp$ are trivial or isomorphic to $\cN\bar\otimes L(\bbF_\infty)$ for some finite factor $\cN$. Abbreviate $p(x) \cM p(x) \oplus p(x)^\perp \cM p(x)^\perp$ with $\cM(x)$. We can thus apply Theorem~\ref{thm:popatakesaki} and fix, for every $x \in X$, a continuous map
\begin{equation} \label{eq:alpha}
\alpha^x \colon [0,\infty) \times \cU(\cM(x)) \to \cU(\cM(x))
\end{equation}
satisfying all the conditions therein. 

In order to show that  $\{ \Phi(x) \}_{x \in X}$ is equi-$\LC^n$ for all $n \in \bbN$, fix $n \in \bbN$, take some $x \in X$, some $u \in \Phi(x)$ and consider the open ball $B_\e(u)$. Fix $y\in X$ and suppose that  $f \colon S^n \to B_{\e/2}(u) \cap \Phi(y)$ is continuous. Fix some $z_0 \in S^n$ and consider the function
\begin{align*}
 f_0 \colon S^n &\to \cU(\cM) \cap \{p(y) \}'\\
  z &\mapsto f(z_0)^* f(z)
\end{align*}  
  By construction $f_0(z_0) = 1$ and the diameter of the image of $f_0$ is smaller
than $\e/2$. Define, after identifying $\cU(\cM) \cap \{p(y) \}'$ with $\cU(\cM(y))$, and using $\alpha^y$ as defined in \eqref{eq:alpha}
\begin{align*}
F_0 \colon S^n \times [0, \infty) &\to \cU(\cM) \cap \{p(y) \}'\\
(z,t)  &\mapsto \alpha_t^y(f_0(z))
\end{align*}

By definition of $\alpha^y$, $F_0$ restricts to $f_0$ on $S^n \times \{0 \}$, and $F_0$ can be moreover continuously extended to $S^n \times [0,\infty]$ by setting $F_0(z, \infty)= 1$. Note finally that, by the construction in  \cite[Theorem 1]{PopaTakesaki:contractible}, the function $\alpha^y$ is such that $\alpha^y_t(1) = 1$ for all $t \in [0, \infty)$. For $z \in S^n$ and $t > 0$ we therefore get, using item \eqref{item4:et} of Theorem~\ref{thm:popatakesaki} at the second step, that
\[
\| F_0( z, t) - 1 \|_{2, \tau} = \| \alpha_t^y(f_0(z)) - \alpha_t^y(1) \|_{2, \tau} \le e^{-t} \|  f_0(z) - 1 \|_{2, \tau} < e^{-t} \cdot \e/2.
\]
The function $F \coloneqq f(z_0) F_0$ is thus a continuous extension of $f$ to $S^n \times [0, \infty]$ whose range is contained in $B_\e(u) \cap \Phi(y)$.
\end{proof}

By Claims \ref{claim:lsc}, \ref{claim:equiLC0} and \ref{claim:equiLCn}, we can apply Theorem \ref{thm:continuous_comparison} under either of the assumptions \ref{item:cov1} or \ref{item:Finfty}. It then follows that there is a continuous function $v \colon X \to \cU(\cM)$ such that $v(x) p(x) v(x)^* = q(x)$, for all $x \in X$.
This completes the proof of \eqref{item:equivalence}. 

Next, we prove \eqref{item:subequivalence} from the statement of the theorem. Assume that $p, q \colon X \to \cM$ are projection-valued $\|\cdot \|_{2,\tau}$-continuous functions such that 
$\tau(p(x)) \le \tau(q(x))$ for all $x \in X$.
Let $(e_{ij})_{i,j=0,1}$ be a system of matrix units for the algebra $M_2$ of $2\times2$ matrices, and let $\text{tr}_2$ be the normalized trace on~$M_2$.  Consider the projection-valued continuous map
\begin{align*}
q \otimes e_{11} \colon X &\to \cM \otimes M_2 \\
x & \mapsto q(x) \otimes e_{11}
\end{align*}

Fix a unital copy of $L^\infty([0,1], \mu)$ in $\cM$, with $\mu$ being the measure on $X$ induced by the restriction of $\tau$ on $L^\infty([0,1], \mu)$, and consider the continuous function
\begin{align*}
r \colon X &\to L^\infty([0,1], \mu) \\
x &\mapsto \chi_{[0, \tau(q(x) - p(x))]}
\end{align*}
where $\chi_{[a,b]}$ denotes the characteristic function on the interval $[a,b]$. Define finally
\begin{align*}
p \oplus r \colon X &\to \cM \otimes M_2 \\
x &\mapsto \begin{pmatrix} p(x) & 0 \\ 0 & r(x) \end{pmatrix}
\end{align*}

This is a projection in the trivial \wstar-bundle. 
By construction we have $\tau \otimes \text{tr}_2(q(x) \otimes e_{11}) = \tau \otimes \text{tr}_2(p(x) \oplus r(x))$ for all $x \in X$. Apply item \eqref{item:equivalence} of the present theorem---note that $\cM \otimes M_2$ satisfies the assumption in item \ref{item:Finfty} if~$\cM$ does, since by \cite{Radulescu:Finfty} the fundamental group of $\cM$ is full and therefore $\cM \otimes M_2 \cong \cM$---to build a continuous map $u \colon X \to \cU(\cM \otimes M_2(\bbC))$ such that
\[
u(x) (p(x) \oplus r(x)) u(x)^* = q(x) \otimes e_{11},\quad x \in X.
\]
For every $i,j = 0,1$ there are continuous functions $u_{ij} \colon X \to \cM$ such that $u = (u_{ij})_{i,j=0,1}$. It follows that $v \colon X \to \cM$ defined as $v(x) := u_{00}(x) p(x)$ pointwise implements the subequivalence between $p$ and $q$.
\end{proof}

Theorem \ref{thm:popatakesaki} also applies to McDuff factors, and this can be used in our argument to prove Theorem \ref{thm:continuous_comparison} also in case $X$ has finite covering dimension and $\cM$ is McDuff. Note however that something stronger than this has already been obtained in \cite[Theorem E]{CCEGSTW}: using methods based on \emph{complemented partitions of unity}, the aforementioned result shows that the statement of Theorem \ref{thm:continuous_comparison} holds for all compact Hausdorff spaces $X$ (with no restriction on the covering dimension) if $\cM$ is McDuff and even if it has property $\Gamma$.

\section{The Trace Problem for Trivial $W^\ast$-bundles} \label{S.4}
This section is devoted to the proof of Theorem \ref{thm:trace_problem}. Our argument follows \cite{Evington:trace}. More precisely, our Proposition \ref{prop:projection} and Theorem \ref{thm:rr0} can be compared with \cite[Theorem 3.2, Theorem 3.7]{Evington:trace} respectively. We assume that the reader is familiar with the basics of the theory of Cuntz subequivalence (see \cite[\S 2--4]{Rordam:UHF} or \cite{Thiel:notes} for an introduction to this topic).

In the remaining part of this section, for positive elements $a$ and $b$  we abbreviate $ab = ba = a$ (meaning that $b$ is a unit for $a$) as $a \vartriangleleft b$. Finally, given $0 \le \e_0 < \e_1 \le1$, we let $h_{\e_0, \e_1} \colon [0,1] \to [0,1]$ be the continuous function which is constantly 0 on $[0, \e_0]$, constantly 1 on $[\e_1, 1]$, and linear on $[\e_0, \e_1]$.

\begin{proposition} \label{prop:projection}
Let $(\cM, \tau)$ be a II$_1$ factor and let $X$ be a compact Hausdorff space with covering dimension at most 1. Fix $a \in C_\sigma(X, \cM)_+$ and let $f \colon X \to [0,1]$ be a continuous function such that $f(x) \le d_{\tau}(a(x))$ for all $x \in X$. Then there exists a projection $p \in C_\sigma(X, \cM)$ such that
\begin{enumerate}
\item $\tau(p(x)) = f(x)$ for all $x \in X$,
\item every $b \in C_\sigma(X, \cM)_+$ that satisfies $a \vartriangleleft b$, also satisfies $p \vartriangleleft b$.
\end{enumerate}
\end{proposition}
\begin{proof}
If $a= 0$ let $p = 0$, otherwise we can assume that $\| a \| = 1$. Given $x \in X$, let $s_x \in \cM$ be the spectral projection $\chi_{(0,1]}(a(x))$. Consider the map $\Phi \colon X \to \cP(\text{Proj}(\cM))$ defined as
\[
\Phi(x) \coloneqq \{ q \in \text{Proj}(\cM) : \tau(q) = f(x) \text { and  } q \in s_x \cM s_x \}.
\]
The codomain of $\Phi$ is the power set of $\text{Proj}(\cM)$, and the latter is equipped with the complete metric induced by $\| \cdot \|_{2, \tau}$.

Clearly $\Phi(x)$ is closed for all $x$. 
\begin{claim}
The function $\Phi$ is lower semicontinuous.
\end{claim}
\begin{proof}
Let $U \subseteq \text{Proj}(\cM)$ be an open set and let $x \in X$ be such that $\Phi(x) \cap U \not = \emptyset$. Let $q \in \Phi(x) \cap U$ and find $1 > \e > 0$ so that $B_\e(q) \subseteq U$. Fix $0 < \delta <  \e/17$.  Since
\[
\| a(x)^{1/k} - s_x \|_{2, \tau} \to 0 \text{ for } k \to \infty, 
\]
we can fix $k \in \bbN$ large enough to have $\| a(x)^{1/k} - s_x \|_{2, \tau} < \delta^2$. 

 Using  continuity of $f$ and $a$, fix an open neighborhood $Z$ of $x$ in $X$ such that
\begin{equation} \label{eq:continuity}
| f(x) - f(y) | < \delta \text{ and } \| a^{1/k}(x) - a^{1/k}(y) \|_{2, \tau} < \delta^2, \quad y \in Z.
\end{equation}
Since $a$ and $q$ are contractions and $\| bc \|_{2, \tau} \le \| b \| \| c\|_{2,\tau}$ for all $b, c \in \cM$, for every $y \in Z$ we have that
\begin{align} \label{eq:continuity1}
\|a(y)^{1/k} q a(y)^{1/k}-a(x)^{1/k} q a(x)^{1/k}\|_{2,\tau}&\leq 2\|a(y)^{1/k}-a(x)^{1/k}\|_{2,\tau} \\ & <2\delta^2. \nonumber
\end{align}
We moreover have $q s_x=q$, which implies
\begin{align} \label{eq:continuity2}
\|a(x)^{1/k} q a(x)^{1/k} - q\|_{2,\tau} &= \|a(x)^{1/k} q a(x)^{1/k} - s_xqs_x\|_{2,\tau} \\ & \leq 2\|s_x-a(x)^{1/k}\|_{2,\tau}\nonumber \\ &<2\delta^2. \nonumber
\end{align}

Fix $y \in Z$. The inequalities in \eqref{eq:continuity1} and \eqref{eq:continuity2}  imply
\[
\| q - a(y)^{1/k} q a(y)^{1/k} \|_{2,\tau} < 4 \delta^2,
\]
thus by \cite[Lemma  XIV.2.2]{Takesaki:III} there is a projection $r \in s_y\cM s_y$ such that
\begin{equation} \label{eq:qr}
\| q - r \|_{2,\tau} < 2\sqrt{12} \delta < 8 \delta.
\end{equation}

The latter inequality, combined with \eqref{eq:continuity} and with $\tau(q) = f(x)$, entails $| \tau(r)  - f(y) | < 9 \delta$. If $\tau (r) > f(y)$, since $ s_y \cM s_y$ is a II$_1$ factor, there are projections $r_0, r_1 \in s_y \cM s_y$ such that $r = r_0 + r_1$ where $\tau(r_0) = f(y)$ and $\tau(r_1) < 9 \delta$, which gives
\begin{equation} \label{eq:qr0}
\| q - r_0 \|_{2, \tau} < \| q - r \|_{2, \tau} + 9 \delta \stackrel{\eqref{eq:qr}}{<} 17 \delta < \e,
\end{equation}
and thus $r_0 \in \Phi(y) \cap B_\e(q) \subseteq \Phi(y) \cap U$ as desired.

If, on the other hand, $\tau(r) < f(y)$ then, as $f(y) \le d_{\tau}(a(y)) = \tau(s_y)$, there exists a projection $r_2 \in s_y\cM s_y$ orthogonal to $r$ such that $\tau(r_2) = f(y) - \tau(r)$. Therefore, a computation like the one in \eqref{eq:qr0} grants $r + r_2 \in \Phi(y) \cap B_\e(q)$.
\end{proof}

\begin{claim} 
The space $\Phi(x)$ is 0-connected for every $x \in X$ and $\{ \Phi(x) \}_{x \in X}$ is equi-$\LC^0$.
\end{claim}
\begin{proof}
Given $x \in X$, any two elements $p,q \in \Phi(x)$ are equivalent projections of $s_x \cM s_x$, and hence there is $u \in \cU(s_x \cM s_x)$ such that $u p u = q$. Any continuous path of unitaries in $s_x \cM s_x$ joining $u$ to 1 automatically yields a continuous path of projections in $s_x \cM s_x$ joining $p$ to $q$, showing that $\Phi(x)$ is 0-connected.

To see that $\{ \Phi(x) \}_{x \in X}$ is equi-$\LC^0$, arguing like at the beginning of Claim \ref{claim:equiLC0}, it is sufficient to prove that for every $\e >0$ there is $\delta > 0$ such that if $p,q \in \Phi(x)$ for some $x \in X$ with $\| p - q \|_{2, \tau} < \delta$, then there exists a path in $\Phi(x)$ joining $p$ to $q$ of diameter smaller than $\e$. It is crucial that $\delta$ depends \emph{neither} on $x$ \emph{nor} on $p$ and $q$.

Given $\e > 0$, let $\delta < \frac{\e}{4 \sqrt{2}}$, fix $x \in X$, and let $p,q \in \Phi(x)$ be such that $\| p - q \|_{2, \tau} < \delta$. By \cite[Lemma XIV.2.1]{Takesaki:III} there is $u \in \cU(s_x\cM s_x)$ such that $upu^* = q$ and $\| 1 - u \|_{2, \tau} < \e /2$. By Lemma \ref{lemma:ulc} there is thus a continuous path $(u_t)_{t \in [0,1]}$ in $\cU(s_x\cM s_x)$ from $s_x$ to $u$  of diameter smaller than $\e /2$. It then follows that $(u_t p u_t^*)_{t \in [0,1]}$ is a continuous path from $p$ to $q$ in $\Phi(x)$ of diameter smaller than $\e$. 
\end{proof}

The claims permit us to apply Theorem \ref{thm:continuous_comparison} to $\Phi$, hence there exists a continuous function $p \colon X \to \text{Proj}(\cM)$ such that $\tau(p(x)) = f(x)$ and $p(x) \in s_x \cM s_x$, for all $x \in X$. Suppose now that $b \in C_\sigma(X, \cM)_+$ is such that $a \vartriangleleft b$. This means in particular that $a(x) \vartriangleleft b(x)$, which in turn implies $s_x \vartriangleleft b(x)$ for all $x \in X$. Since $p(x)  \vartriangleleft s_x$, it follows that $p(x) \vartriangleleft b(x)$ for all $x \in X$, and thus $p \vartriangleleft b$, since multiplication is defined pointwise on $C_\sigma(X, \cM)$.
\end{proof}

In the remaining part of this section, given a positive contraction $a \in A$ and $\e > 0$, we use the notation $(a - \e)_+$ to abbreviate $f_\e(a)$, where $f \colon [0,1] \to [0,1]$ is defined as $f_\e(t) \coloneqq \max \{ 0, t -\e \}$. 

\begin{theorem} \label{thm:rr0}
Let $(\cM, \tau)$ be a II$_1$ factor and let $X$ be a compact Hausdorff space with covering dimension at most 1.  For every $a \in  C_\sigma(X, \cM)_+$ and $\e > 0$ there is a projection $p \in  C_\sigma(X, \cM)$ such that $(a- \e)_+ \precsim p \precsim a$.
\end{theorem}
\begin{proof}
This argument is analogous to one used to prove \cite[Theorem 3.7]{Evington:trace}. We briefly sketch it for the reader's convenience.

We use the abbreviation $\cN \coloneqq  C_\sigma(X, \cM)$. If $a = 0$, take $p =0$. Otherwise, we can assume $\| a \| = 1$. Using the function $h_{\e_0,\e_1}$ defined in the paragraph preceding  Proposition~\ref{prop:projection} define $f \colon X \to [0,1]$ by $f(x) \coloneqq \tau(h_{\e/2, \e}(a(x)))$ for $x \in X$. Then
\[
d_\tau((a(x) - \e)_+) \le f(x) \le d_\tau((a(x) - \e/2)_+), \quad x \in X.
\]

Since $(a- \e/2)_+ \vartriangleleft h_{0, \e/2}(a)$, by Proposition \ref{prop:projection} there is a projection $p \in \cN$ such that $\tau(p(x)) = f(x)$ for all $x \in X$ and $p \vartriangleleft h_{0,\e}(a)$. Since $h_{0,\e}(a) \in \overline{a \cN a}$ this implies $p \in \overline{a \cN a}$ and thus $p \precsim a$.

Set $b \coloneqq (a - \e)_+$ and fix $\delta > 0$. Then $\tau(h_{0, \delta}(b(x))) \le d_\tau(b(x))$ for all $x \in X$, and therefore
\[
d_\tau(1- h_{0, \delta}(b(x))) \ge \tau ( 1 - h_{0,\delta}(b(x))) \ge 1 - f(x), \quad x \in X.
\]
By Proposition \ref{prop:projection} there is a projection $q \in \cN$ such that
\begin{equation} \label{eq:trace_value}
\tau(q(x)) = 1 - f(x), \quad x \in X,
\end{equation}
and
\[
q \vartriangleleft 1 - h_{0, \delta}(b) \vartriangleleft 1 - h_{\delta, 2\delta}(b).
\]
It follows that $h_{\delta, 2\delta}(b) \vartriangleleft q^\perp$, which in turn gives $(b - 2\delta)_+ \precsim h_{\delta, 2 \delta}(b) \precsim q^\perp$.

By Theorem \ref{thm:continuous_comparison} and \eqref{eq:trace_value}, repeating this argument for different values of~$\delta$ will always return projections that are unitarily conjugate, and thus Cuntz equivalent, to $q$. This shows that $(b - 2 \delta) \precsim q^\perp$ for all $\delta > 0$, which in turn implies $b \precsim q^\perp$, that is $(a- \e)_+ \precsim q^\perp$.

Since $\tau(p(x)) = \tau(q^\perp(x)) = f(x)$ for all $x \in X$, Theorem \ref{thm:continuous_comparison} ensures that $p$ and $q^\perp$ are unitarily conjugate in $\cN$, and therefore $(a-\e)+ \precsim p \precsim a$.
\end{proof}

\begin{proof}[Proof of Theorem \ref{thm:trace_problem}] Fix a compact Hausdorff space $X$ with covering dimension at most 1 and a II$_1$ factor $\cM$. 
	We need to prove that the trace problem has positive solution for $C_\sigma(X, \cM)$. By Proposition \ref{prop:comparison}, it suffices to prove that $C_\sigma(X, \cM)$ has strict comparison relative to $X$.

By \cite[Lemma 2.14]{Evington:trace}, it is sufficient to verify strict comparison relative to $X$ for $a,b \in M_n(C_\sigma(X, \cM))_+$, for $n \in \bbN$. Suppose then that there is $\gamma > 0$ such that  $d_\tau(a(x)) \le (1 -\gamma) d_\tau(b(x))$ for all $x \in X$. Note that $M_n(C_\sigma(X, \cM)) \cong C_\sigma(X, M_n(\cM))$, hence $M_n(C_\sigma(X, \cM))$ is again a trivial $W^\ast$-bundle with base space of dimension at most 1. Because of this we can assume, without loss of generality, that $a, b \in C_\sigma(X, \cM)$.

Fix $\e > 0$. By \cite[Proposition 3.3]{NgRobert:commutators}, applied to $a$ as an element in $C_\sigma(X, \cM)$, there is $\delta > 0$ such that
\[
d_\tau((a(x) - \e)_+) \le  (1 - \gamma/2)d_\tau((b(x) - \delta)_+), \quad x \in X.
\]
By Theorem \ref{thm:rr0} there are projections $p,q \in C_\sigma(X, \cM)$ such that
\[
(a - 2\e)_+ \precsim p \precsim (a - \e)_+ \text{ and } (b- \delta)_+ \precsim q \precsim b.
\]
We therefore have
\[
\tau(p(x)) \le d_\tau((a(x) - \e)_+) \le (1 - \gamma/2) d_\tau((b(x) - \delta)_+) \le \tau(q(x)), \quad x \in X.
\]
Theorem \ref{thm:continuous_comparison} then implies that $p \precsim q$, and thus $(a - 2\e)_+ \precsim p \precsim q \precsim b$. As $\e$ is arbitrary, we conclude that $a \precsim b$.
\end{proof}

\section{Concluding remarks} \label{S.5}
The application of continuous selection principles to the study of factors can be traced back to \cite{PopaTakesaki:contractible}, where Michael's results from \cite{michael1959convex} are used to prove the existence of continuous cross-sections for quotients of groups of unitaries and to prove that the automorphism group of the hyperfinite II$_1$ factor is contractible (the proof of the first statement contains a gap that was recently fixed in \cite{Ozawa:contract} using \cite{Michael:selection1}). Michael's selection theorems appear to be tailor-made for the analysis of \wstar-bundles, and we briefly discuss here some possible future directions of research related to the results presented in this note.

The first natural question is whether Theorem \ref{thm:continuous_comparison} could be proved with no assumption on the II$_1$ factor $(\cM, \tau)$ and on the compact Hausdorff space~$X$. The case where~$X$ has infinite covering dimension would need to be approached with different tools than those used in this paper, as the selection principle in Theorem \ref{thm:michael_selection} is intrinsically limited to finite-dimensional spaces. If, on the other hand, $X$ has covering dimension smaller that $n+1$, the first step towards adapting our arguments would be showing that the unitary groups of the corners of $\cM$ are $n$-connected. After the submission of the first version of this manuscript, this has been proved to be true (in the separable case) in the recent preprint \cite{Jekel:contract}, where Jekel shows that the unitary group of every separably representable II$_1$ factor is contractible in the strong operator topology. We point out, however, that Jekel's methods do not seem to be directly applicable to prove that the family $\{\Phi(x) \}_{x \in X}$ in the proof of  Theorem \ref{thm:continuous_comparison} is equi-LC$^n$.

Another direction worth exploring further is whether Theorem \ref{thm:continuous_comparison}  could be used to solve the trace problem in case of trivial $W^\ast$-bundles whose base space has finite covering dimension and in case  of $\cM \cong \cN \bar \otimes L(\bbF_\infty)$ for some finite factor $\cN$. If $\Phi \colon X \to \cP(\text{Proj}(\cM))$ is as in the proof of Proposition~\ref{prop:projection} and $\cM$ is as above, then Theorem \ref{thm:continuous_comparison} suffices to guarantee that $\Phi(x)$ is $n$-connected for every $n \in \bbN$, but it is not obvious whether the family is $\{ \Phi(x) \}_{x \in X}$ is equi-LC$^n$.

A different problem is whether 
 Michael's continuous selection principles could be used to obtain analogues of Theorems \ref{thm:continuous_comparison} and \ref{thm:trace_problem} for \emph{non-trivial} $W^\ast$-bundles (see \cite{Ozawa:dixmier} or \cite[\S 3.6]{CCEGSTW} for the definition of~$W^\ast$-bundle). More precisely, given a~$W^\ast$-bundle $\cN$ with base space $X$, to each point $x \in X$ corresponds a \emph{fiber}~$\cM_x$, isomorphic to $\pi_{\rho_{\delta_x}}(\cN)''$ (in fact to  $\pi_{\rho_{\delta_x}}(\cN)$, see \cite[Theorem~11]{Ozawa:dixmier}), where~$\rho_{\delta_x}$ is the trace on $\cN$ corresponding to the Dirac measure of $x \in X$ as described in \eqref{eq:measure}. It is possible to define a topology on the disjoint union $B \coloneqq \bigsqcup_{x \in X} \cM_x$ so that $\cN$ is isomorphic to the set of continuous \emph{sections} $f \colon X \to B$ so that $f(x) \in \cM_x$ for all $x \in X$ (this is done in detail in \cite[\S 3]{EvingtonPennig} and in \cite[\S 3.6]{Evington:phd}, which adapt to von Neumann algebras the theory of Banach bundles developed in \cite{FellDoran}).

 This might appear as a setup suitable for Michael's continuous selection theorem, setting in particular $Y = B$ in the statement of Theorem \ref{thm:michael_selection}. It is however crucial for Theorem \ref{thm:michael_selection} that $Y$ is a complete metric space. One can then assume that $X$ is metrizable, in which case $B$ can be proved to be metrizable itself, but even in this case it is not clear how to find a complete metric on~$B$ compatible with the topology.

Nevertheless, we point out that one Michael's selection principle from \cite{Michael:selection1} has been successfully adapted to bundles of Banach spaces in \cite{Lazar}, suggesting that a similar adaptation of Theorem \ref{thm:michael_selection}, and of the results in \cite{Michael:selection2}, to bundles of Banach spaces and of von Neumann algebras is plausible.

\bibliographystyle{plain}
\bibliography{bibmichael}

\end{document}